\documentclass{article}
\usepackage{graphicx} 
\usepackage[affil-it]{authblk}
\usepackage[utf8]{inputenc}
\usepackage[english]{babel}
\usepackage{titlesec}
\usepackage[margin=0.5in]{geometry}
\usepackage{mathtools}
\usepackage{enumerate}
\usepackage{enumitem}
\usepackage{lineno} 
\usepackage{adjustbox}
\usepackage{multirow}
\usepackage{chngcntr}
\usepackage{titlesec}
\usepackage[utf8]{inputenc}
\usepackage{mathrsfs}
\usepackage{lscape}
\usepackage{cite}
\usepackage{cite}

\usepackage{enumitem}
\usepackage{subfigure}
\usepackage{subcaption}

\usepackage[hyphens]{url}
\usepackage{hyperref}
\usepackage{float}
\usepackage[section]{placeins}
\usepackage[dvipsnames]{xcolor}
\usepackage{soul}
\usepackage{amsmath}
\usepackage[labelfont=bf,font=small]{caption}
\usepackage{bm}
\usepackage[title]{appendix}
\usepackage{graphicx}
\usepackage{tabularx}
\usepackage{float}
\usepackage{mathtools}
\usepackage{amsmath,amssymb,amsthm}
\usepackage{oubraces}
\numberwithin{equation}{section}
\usepackage{chngcntr}
\usepackage{mathrsfs}
\usepackage{multirow}
\usepackage{chngcntr}
\usepackage{parskip}
\usepackage{accents}
\usepackage{soul}
\usepackage{orcidlink}

\usepackage{cleveref}

\usepackage[normalem]{ulem}

 \usepackage{changepage}

\usepackage[dvipsnames]{xcolor}   

\definecolor{dgn}{rgb}{0.0, 0.5, 0.0}

\newcommand{\allblack}{\color{black}{}}

\newtheorem{theorem}{Theorem}
\newtheorem{prop}{Proposition}

\usepackage{lipsum}
\makeatletter
\g@addto@macro{\endabstract}{\@setabstract}
\newcommand{\authorfootnotes}{\renewcommand\thefootnote{\@fnsymbol\c@footnote}}%
\makeatother

\begin{document}
\title{Dead Zones Enable Flexible Phase Organization in Coupled Oscillators}
\author{Naghmeh Akhavan \orcidlink{0000-0002-9474-4486}$^{a}$, Ruby Kim \orcidlink{0000-0003-0209-4875}$^{b}$\\
{\it {\small $^{a}$  Department of Mathematics, University of Michigan, Ann Arbor, MI}}

 {\it {\small $^{b}$  Department of Anesthesiology, University of Michigan, Ann Arbor, MI}}


\vspace{-1.5cm}
    }
  \date{}
\maketitle
\vspace{1cm}
\begin{center}
    \textbf{Abstract}
\end{center}
\noindent 
Coupled oscillator networks underlie many biological systems, from cardiac cycles to circadian rhythms. Phase-reduced models such as the Kuramoto model have been widely used to study synchronization, but {they typically} assume {that} oscillators {remain} continuously responsive to inputs
{and often produce}
tightly clustered phase distributions. Biological oscillators, however, 
{commonly}
exhibit phase intervals during which inputs have little or no effect, called ``dead zones." Here, we extend the Kuramoto model 
{by introducing}
receiver-gated dead zones, in which oscillators transiently ignore incoming signals. Using analytical and numerical approaches, we show that dead zones can reduce synchronization rates,
modulate the distribution of phase-locked solutions, 
and
{modify the}
stability of phase-locked states. For identical oscillators, 
{full synchrony remains locally exponentially stable and numerically dominant}, although convergence times 
{depend} sensitively {on the dead-zone width.} 
For heterogeneous oscillators 
{that admit}
phase-locked solutions when the coupling strength {satisfies} $K>K^*$, dead zones broaden the long-term phase distributions. 
Numerical 
{exploration across} 
dead zone widths {reveals a transition in which steady phase gaps lose stability,}
resulting in dead zone-induced phase drifting. Overall, these results identify phase-response dead zones as a biologically plausible mechanism for flexible phase organization beyond classical Kuramoto dynamics.
\section{Introduction}

Populations of coupled oscillators are ubiquitous in biology, {arising in systems ranging from circadian clocks and cell-cycle regulation to developmental and gene-regulatory dynamics~\cite{heltberg2021tale, jimenez2022principles}}.
Networks of units, each with its own intrinsic dynamics and coupled to one {another}
, collectively regulate essential physiological processes such as cardiac rhythms, hormonal cycles, and circadian ($\sim$24-hour) timing. Despite substantial variability at the level of individual units, these systems often exhibit coherent macroscopic behavior. Understanding how such collective signals emerge from heterogeneous, interacting components is both mathematically interesting and biologically important. 

One example of such a system is the suprachiasmatic nucleus (SCN) in the mammalian brain, which contains a population of circadian neurons that receive light input from the retina, coordinate with each other, and relay timing signals to other brain regions \cite{hastings2018generation,brown2009spatiotemporal}. Each circadian neuron has its own molecular clock, a transcription-translation feedback loop (TTFL), that drives nearly 24-hour periodic rhythms in gene expression \cite{hurley2016circadian}. While detailed mathematical models have been developed to describe intracellular circadian dynamics   \cite{asgari2019mathematical,tyson2008biological,gonze2011modeling}, there has been less of a focus on intercellular coupling and population heterogeneity, which are essential for robust timekeeping at the tissue level \cite{collins2020circadian,dewoskin2014not}. 

Phase-reduced models provide a natural framework for studying these dynamics. In particular, Kuramoto-type models, which represent each oscillator by its phase and interaction strength, have been widely used to conceptualize network-level interactions and derive low-dimensional descriptions of macroscopic behavior \cite{acebron2005kuramoto,hannay2019macroscopic}. However, a key assumption underlying classical phase models is that oscillators are continuously responsive to inputs from others. As a consequence, {in regimes of sufficiently strong attractive coupling, }classical Kuramoto-type systems 
{often}
exhibit {highly coherent behavior,}
with {oscillator phases remaining relatively}
tightly clustered. 

In reality, biological oscillators often exhibit strongly phase-dependent responsiveness \cite{johnson1992phase}. Notably, many experimentally observed phase response behaviors contain ``dead zones" during which inputs have little or no effect \cite{lewy2010clinical}, which are hypothesized to modulate entrainment to external cues \cite{pfeuty2011robust}. {For instance, light pulses typically phase-shift circadian rhythms in rodents, except during dead zone intervals when there is no phase response \cite{daan1976functional}. A dead zone can also be observed in the efficacy of melatonin, where melatonin administered during the dead zone has little to no effect on circadian phase \cite{cipolla2018melatonin}. Such features may arise naturally from biological mechanisms including saturation in biochemical signaling pathways \cite{uriu2019saturated,best2024mathematical} and changes in the electrical responsiveness of neurons \cite{komal2026iprgc}.} Despite extensive study of phase response curves (PRCs) at the single-oscillator level, their implications for network-level behavior remain poorly understood \cite{smeal2010phase}.

In this study, we modify the Kuramoto model by introducing a receiver-gated dead zone for each oscillator. In this formulation, oscillators transiently ignore incoming signals when their phase lies within prescribed intervals representing biological dead zones. We show that this modification meaningfully alters the collective dynamics, slowing the approach to synchronization and promoting phase dispersion. As a result, we identify phase response dead zones as a plausible mechanism for flexible phase organization, enabling broad phase distributions consistent with biological systems.

{More specifically, for identical oscillators, we show that the receiver-gated model retains a Lyapunov structure and that full synchrony remains locally exponentially stable, although the rate of convergence depends sensitively on dead-zone width. 
For heterogeneous oscillators, we derive the corresponding modification to the locked-frequency relation induced by receiver-side gating. 
Finally, through numerical simulations, we show that dead zones can alter the critical coupling threshold and induce phase drifting.}




\section{Receiver-gated Kuramoto model}\label{sec:receiver_gated_model}
In this section, we introduce a modified Kuramoto model of interacting biological oscillators in which each oscillator's response to others' phases depends on where it is in its cycle. 
This receiver-gated formulation is motivated by 
experimental PRCs which often show intervals of weak or negligible responsiveness, commonly called dead zones \cite{lewy2010clinical,pfeuty2011robust}.
To capture this feature at the network level, we allow each oscillator's receptivity to incoming signals to vary through a phase-dependent gate. 

\subsection{Model formulation}
We consider a population of $N$ coupled phase oscillators with phases $\theta_i(t) \in \mathbb T = \mathbb R/2\pi \mathbb Z$, for $i=1, \dots, N$. 
The receiver-gated Kuramoto model is given by
\begin{align}\label{eq:receiver_gated_general}
    \dot \theta_i = \omega_i + \frac{K}{N} S(\theta_i) \sum_{j=1}^N \sin(\theta_j - \theta_i), \qquad i=1, \dots, N, 
\end{align}
where $\omega_i$ is the natural frequency of oscillator $i$, $K>0$ is the coupling strength, and $S: \mathbb T \to [0,1]$ is a $2\pi$-periodic receiver gate. 
The factor $S(\theta_i)$ determines how strongly oscillator $i$ responds to the collective influence of the population at its current phase. 

When $S(\theta_i) \approx 1$, oscillator $i$ is highly receptive to the population and responds almost as in the classical Kuramoto model. 
When $S(\theta_i) \approx 0$, its sensitivity is strongly suppressed so its motion is driven by its intrinsic frequency. 
Accordingly, the gate acts as a phase-dependent receptivity function and represents the idea that an oscillator may temporarily ignore the population {activity} when it lies in an insensitive phase interval. 


The interaction kernel in \Cref{eq:receiver_gated_general} is the classical sine coupling, which is odd (i.e., $\sin(-\phi)=-\sin(\phi)$) and $2\pi$-periodic.
This form of coupling promotes alignment of phase when the coupling is attractive ($K>0$). 
If $S(\theta) \equiv 1$, then \eqref{eq:receiver_gated_general} reduces to the standard all-to-all Kuramoto model 
\begin{align}\label{eq:classical_kuramoto}
    \dot \theta_i = \omega_i + \frac{K}{N} \sum_{j=1}^N \sin(\theta_j - \theta_i), \qquad i=1, \dots, N,
\end{align}
where $\theta_j>\theta_i$ causes $\theta_i$ to speed up towards $\theta_j$ and $\theta_i>\theta_j$ causes $\theta_i$ to slow down.
Thus, the receiver-gated system in \Cref{eq:receiver_gated_general} can be viewed as a state-dependent extension of the classical Kuramoto model in which the effective coupling is modulated by the phase of the receiving oscillator. 

\subsection{Dead-zone gate definition and properties}

To model an interval of reduced sensitivity, we define the receiver gate using a smooth dead-zone profile. 
Let $\theta_0 \in \mathbb T$ be the center of the dead zone, let $w \in (0, 2\pi)$ denote its width, and let $k>0$ control the sharpness of the transition at the edges. We define gating using a double sigmoid function
\begin{align}\label{eq:GD_def}
    S(\theta)
  = 1-
  \sigma\Big(k\left(d(\theta,\theta_0)+\tfrac{w}{2}\right)\Big)
  \sigma\Big(k\left(-d(\theta,\theta_0)+\tfrac{w}{2}\right)\Big),
\end{align}
where 
\begin{align}
    \sigma(x) = \frac{1}{1+e^{-x}},
\end{align}
is the logistic function and 
\begin{align}\label{eq:wrapped_distance}
    d(\theta,\theta_0)=\operatorname{atan2}(\sin(\theta-\theta_0), \cos(\theta-\theta_0))\in(-\pi,\pi]
\end{align}
{denotes the wrapped angular distance from \(\theta_0\) to \(\theta\), that is, the shortest signed displacement on the circle. Here \(\operatorname{atan2}(y,x)\) returns the angle of the vector \((x,y)\) in \((-\pi,\pi]\), which ensures that \(d(\theta,\theta_0)\) correctly accounts for periodicity on \(\mathbb T\).}

{This form of the receiver gate is chosen to balance biological interpretability with mathematical and numerical convenience. The function \(S\) is smooth, which facilitates analytical arguments such as linearization and local stability analysis, while for every finite \(k\) it remains strictly between 0 and 1, avoiding discontinuities in the vector field. Moreover, \eqref{eq:GD_def} provides a smooth approximation of an on--off gate: when \(|d(\theta,\theta_0)|<w/2\), the phase \(\theta\) lies inside the dead zone, the product of the two logistic factors is close to one, and \(S(\theta)\) is therefore close to zero; outside this interval, at least one logistic factor is small, so \(S(\theta)\) is close to one. 
In the formal limit \(k\to\infty\), the gate approaches a piecewise constant dead-zone profile.}


To characterize synchronization throughout the paper, we use the standard Kuramoto order parameter
\begin{align}\label{eq:order_parameter_identical}
R(t)e^{i\psi(t)} = \frac{1}{N}\sum_{j=1}^N e^{i\theta_j(t)},
\end{align}
where $R(t) \in [0,1]$ measures the degree of phase coherence. 
The value $R(t)=1$ corresponds to complete synchrony while smaller values indicate widely distributed phases around the circle \cite{acebron2005kuramoto,strogatz2000kuramoto}.




\subsection{Receiver-gated interaction structure}
\Cref{fig:1} summarizes the main features of the receiver-gated interaction and provides a first comparison with the classical Kuramoto model.  All parameter values used throughout the paper are provided in \Cref{table:S1}.
Panel (A) shows a one-dimensional slice of the receiver-gated interaction term, showing how the dead zone suppresses coupling over part of the cycle while preserving the phase-attractive structure outside that interval. 
Panel (B) shows the full interaction term $S(\theta_i)\sin \phi$ as a function of the receiver phase $\theta_i$ and the phase difference $\phi = \theta_j - \theta_i$. 
For receiver phases inside the dead zone, the interaction is strongly attenuated, whereas outside the dead zone the coupling resembles the classical sine interaction. 
\begin{figure}[ht]
    \centering
    \includegraphics[width=1\linewidth]{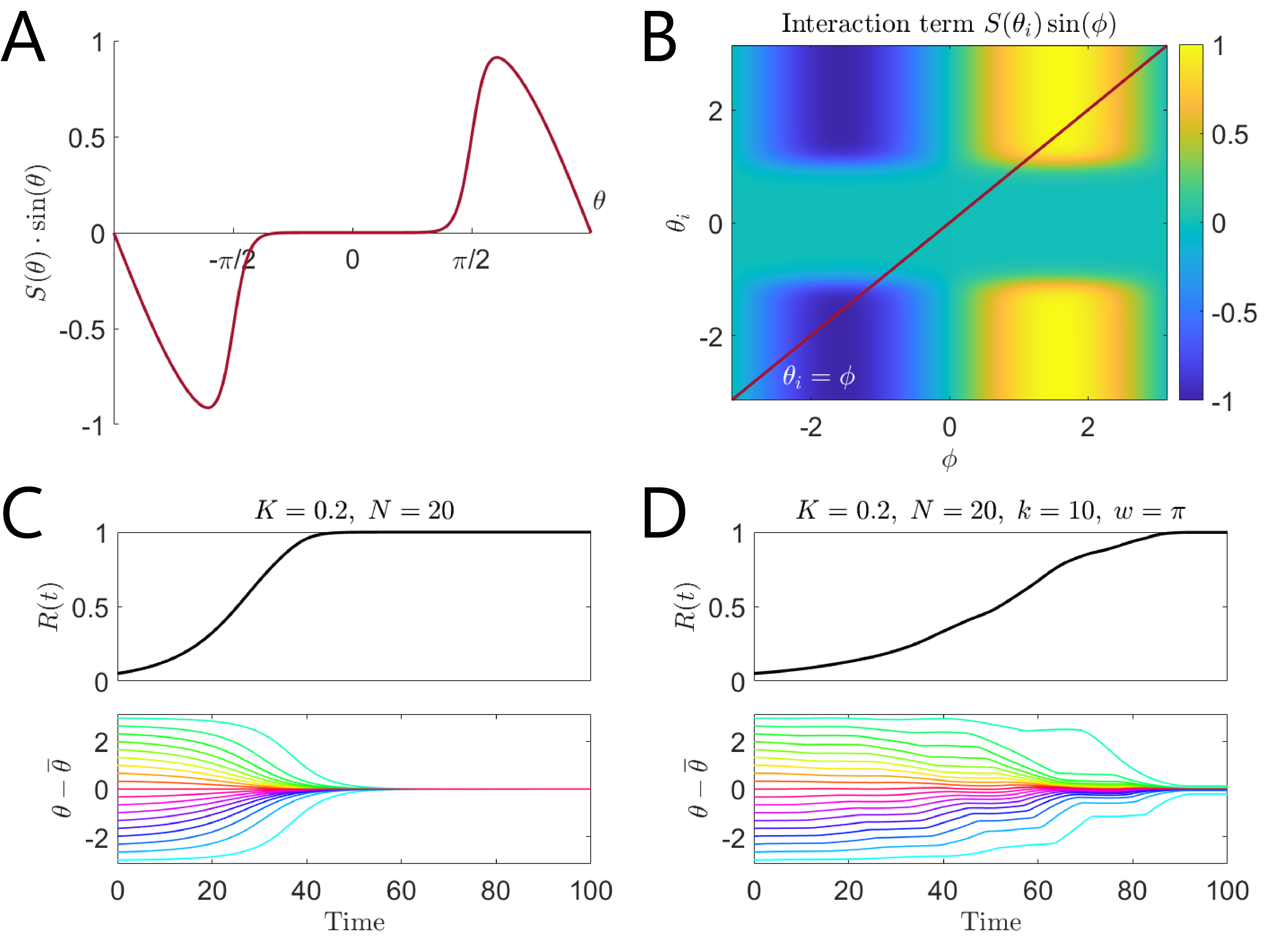}
    \caption{{\bf Receiver-gated coupling.} (A) Slice of receiver-gated coupling function for illustrative purposes. (B) Heatmap of coupling from oscillator $j$ to oscillator $i$ dependent on $\theta_i$ and $\phi=\theta_j-\theta_i$. (C) Model simulations with $N=20$, coupling strength $K=0.2$, and identical frequencies $\omega=2\pi/24$. $R(t)$ is the order parameter and for each oscillator, $\theta - \bar\theta$ represents the phase difference from the mean phase. Initial conditions were uniformly distributed from $0$ to $2\pi$. (D) Model simulations for receiver-gated coupling with dead zone width $w=\pi$ and $k=10$. As in panel C, $K=0.2$ and $N=20$.}
    \label{fig:1}
\end{figure}
Panels (C) and (D) display representative simulations for identical circadian oscillators with a natural period of 24 hours. 
In Panel (C) the classical Kuramoto model exhibits the expected rapid collapse of the phase differences toward synchrony, accompanied by a fast increase of the order parameter $R(t)$ toward one. 
In Panel (D), the receiver-gated model still converges to synchrony but the transient dynamics are visibly slower. 
Phase differences persist longer before collapsing and the growth of $R(t)$ is delayed relative to the ungated case. 
This behavior is consistent with the interpretation of the dead zone as a temporary reduction in receptivity. When oscillators pass through insensitive phases the effective coupling is weakened and the approach to synchrony slows accordingly. 

Based on numerical exploration, we hypothesized that the receiver gating does not destroy the tendency toward synchronization in the identical oscillator setting but rather reshapes the transient dynamics by introducing phase dependent windows of reduced interaction. 
We investigate this behavior in the following analysis in \Cref{sec:identical_analysis}.

\section{Identical oscillators: synchronization analysis}\label{sec:identical_analysis}
Here, we analyze the case of identical oscillators where all natural frequencies are identical. 
Considering the system on a rotating frame, we may assume without loss of generality that $\omega_i \equiv 0$, for $i=1, \dots, N$. 
The receiver-gated Kuramoto model then takes the form 
\begin{align}\label{eq:identical_receiver_model}
\dot{\theta}_i = \frac{K}{N}S(\theta_i)\sum_{j=1}^N \sin(\theta_j-\theta_i),
\qquad i=1,\dots,N.
\end{align}
In this setting, complete synchrony corresponds to a common phase shared by all oscillators, and the main question is whether the dead-zone gate alters the classical tendency toward synchronization. 
Below, we show that the identical frequency system retains a gradient-like structure where the synchronization manifold remains the distinguished stable state, while the receiver gate mainly affects the rate at which trajectories approach it. 

\subsection{Lyapunov analysis}
A useful starting point is the classical Kuramoto potential 
\begin{align}\label{eq:kuramoto_potential}
U(\theta) := -\frac{1}{2N}\sum_{i=1}^N\sum_{j=1}^N \cos(\theta_i-\theta_j),
\end{align}
defined on the $N$-torus $\mathbb T^N$. 
This function measures phase dispersion in the network and is minimized along the synchronized manifold. 

The key observation is that the receiver gated system still admits a Lyapunov description with respect to $U$. 
More precisely, the dynamics can be written as a state-dependent gradient flow in which the gate rescales the dissipation associated with each oscillator. 

\begin{theorem}
    For the identical oscillator receiver gated system \eqref{eq:identical_receiver_model}, the potential $U$ satisfies
    \begin{align}\label{eq:Udot_identical}
\dot{U}(\theta(t))
=
-K\sum_{i=1}^N S(\theta_i)
\left(\frac{\partial U}{\partial \theta_i}\right)^2
\leq 0.
\end{align}
Consequently, every trajectory converges to the largest invariant set contained in 
\begin{align}\label{eq:invariant_set_identical}
\mathcal{E}
:=
\left\{
\theta\in\mathbb{T}^N
\;\middle|\;
S(\theta_i)\frac{\partial U}{\partial \theta_i}=0
\text{ for all } i=1,\dots,N
\right\}.
\end{align}
\end{theorem}

\begin{proof}
    Differentiating \eqref{eq:kuramoto_potential} with respect to $\theta_i$ gives 
    \begin{align}\label{eq:dU_dtheta}
\frac{\partial U}{\partial \theta_i}
=
\frac{1}{N}\sum_{j=1}^N \sin(\theta_i-\theta_j).
\end{align}
Since $\sin(\theta_j - \theta_i) = -\sin(\theta_i - \theta_j)$, the dynamics \eqref{eq:identical_receiver_model} can be rewritten as  
\begin{align}\label{eq:gradient_form_identical}
\dot{\theta}_i
=
\frac{K}{N}S(\theta_i)\sum_{j=1}^N \sin(\theta_j-\theta_i)
=
-K\,S(\theta_i)\frac{\partial U}{\partial \theta_i}.
\end{align}
Therefore, 
\begin{align}
\dot{U}(\theta(t))
=
\sum_{i=1}^N \frac{\partial U}{\partial \theta_i}\dot{\theta}_i
=
-K\sum_{i=1}^N S(\theta_i)
\left(\frac{\partial U}{\partial \theta_i}\right)^2
\le 0,
\end{align}
which proves \eqref{eq:Udot_identical}. Since $U$ is continuous on the compact
set $\mathbb{T}^N$, it is bounded below. Thus $U(\theta(t))$ is nonincreasing
along trajectories and LaSalle's invariance principle~\cite{alligood1997chaos} implies convergence to
the largest invariant set contained in $\{\dot U=0\}$, which is precisely
\eqref{eq:invariant_set_identical}.
\end{proof}
{Because the gate is strictly positive for every finite $k$, receiver gating changes the rate of descent of the classical Kuramoto potential, but not its set of critical points. }

This identity shows that the receiver gate does not destroy the dissipative structure of the identical-oscillator model. 
Instead, it modulates the local rate of dissipation through the factor $S(\theta_i)$. 
When an oscillator lies in a phase region where the gate is small its contribution to the decay of the Lyapunov functional is correspondingly reduced. 
In this sense, the dead zone acts not by changing the direction of descent but by slowing the descent in selected parts of phase space. 

For smooth gating introduced in \eqref{eq:GD_def}, one has $S(\theta)>0$ for all $\theta$ whenever $k<\infty$.
In that case, $\dot U=0$ implies $\nabla U =0$, so the invariant set reduces to the critical set of the classical Kuramoto potential. 
This immediately suggests that the receiver gate preserves the equilibrium set of the classical Kuramoto model even though it changes the transient dynamics.

\subsection{Equilibrium characterization}
We next characterize the equilibrium set of \eqref{eq:identical_receiver_model}. 
A configuration
$\theta^\ast\in\mathbb{T}^N$ is an equilibrium if and only if
\begin{align*}
    \dot{\theta}_i=0 \qquad \text{for all } i=1,\dots,N.
\end{align*}
Substituting into \eqref{eq:identical_receiver_model} yields the condition
\begin{align}\label{eq:equilibrium_condition_identical}
S(\theta_i^\ast)\sum_{j=1}^N \sin(\theta_j^\ast-\theta_i^\ast)=0,
\qquad i=1,\dots,N.
\end{align}
Thus, in general, an oscillator may satisfy the equilibrium condition either
because the coupling sum vanishes or because the receiver gate is zero. For the
smooth dead-zone gate considered here, however, the second possibility does not
occur at finite $k$, since
\begin{align*}
    0<S(\theta)<1 \qquad \text{for all } \theta\in\mathbb{T}.
\end{align*}
Accordingly, \eqref{eq:equilibrium_condition_identical} reduces to
\begin{align}\label{eq:classical_equilibrium_condition_identical}
\sum_{j=1}^N \sin(\theta_j^\ast-\theta_i^\ast)=0,
\qquad i=1,\dots,N,
\end{align}
which is exactly the equilibrium condition for the classical identical Kuramoto
model. Equivalently, equilibria are precisely the critical points of the
potential $U$.

The synchronized manifold
\begin{align}\label{eq:sync_manifold_identical}
\mathcal{S} = \{\theta\in\mathbb{T}^N:\theta_1=\cdots=\theta_N \; (\mathrm{mod}\;2\pi)\}
\end{align}
is always contained in this equilibrium set. 
Indeed, if all oscillators share a common phase, then every phase difference vanishes and hence each coupling sum is zero. 
More generally, the equilibrium set additionally contains the same phase-locked states with symmetric phase configurations as in the classical all-to-all Kuramoto model.  
What changes in the receiver gated system is not the location of these equilibria but the way trajectories move toward them. 

For attractive coupling on the complete graph the synchronized manifold is the distinguished stable equilibrium family, whereas {the remaining critical configurations are not expected to be asymptotically stable.}
Combined with the Lyapunov structure above, this indicates that the receiver gate preserves the classical equilibrium geometry of the identical frequency system. 

\subsection{Local exponential stability}
We now examine the local stability of the synchronized manifold and identify how the dead zone modifies the convergence rate near synchrony. 
Let
\begin{align*}
    \theta_i(t) = \theta^\ast+\delta_i(t), \qquad |\delta_i|\ll 1,
\end{align*} 
where $\theta^\ast$ is a synchronized phase and $\delta_i$ represents a small perturbation. 

Using the expansion
\begin{align*}
    \sin(\theta_j-\theta_i)=\sin(\delta_j-\delta_i) = (\delta_j-\delta_i)+\mathcal{O}(\|\delta\|^2),
\end{align*}
we obtain
\begin{align}\label{eq:linear_sine_expansion}
\sum_{j=1}^N \sin(\theta_j-\theta_i) = \sum_{j=1}^N(\delta_j-\delta_i)+\mathcal{O}(\|\delta\|^2).
\end{align}
Similarly,
\begin{align}\label{eq:gate_expansion}
S(\theta_i) = S(\theta^\ast)+S'(\theta^\ast)\delta_i+\mathcal{O}(\|\delta\|^2).
\end{align}
Substituting \eqref{eq:linear_sine_expansion} and \eqref{eq:gate_expansion}
into \eqref{eq:identical_receiver_model} gives
\begin{align}\label{eq:prelinear_identical}
\dot{\delta}_i = \frac{K}{N}
\Big(S(\theta^\ast)+S'(\theta^\ast)\delta_i+\mathcal{O}(\|\delta\|^2)\Big)
\Big(\sum_{j=1}^N(\delta_j-\delta_i)+\mathcal{O}(\|\delta\|^2)\Big).
\end{align}
The term involving $S'(\theta^\ast)\delta_i$ is multiplied by a quantity of order $\mathcal{O}(\|\delta\|)$, so it contributes only at quadratic order.
Therefore, it does not enter the linearization. 
At leading order we obtain
\begin{align}\label{eq:linearized_identical}
\dot{\delta}_i = \frac{K\,S(\theta^\ast)}{N}\sum_{j=1}^N(\delta_j-\delta_i).
\end{align}
Introducing the graph Laplacian $L$ of the complete graph,
\begin{align*}
    L_{ij} =
\begin{cases}
N-1, & i=j,\\
-1, & i\neq j,
\end{cases}
\end{align*}
we may write \eqref{eq:linearized_identical} in vector form as
\begin{align}
\dot{\delta} =
-\frac{K\,S(\theta^\ast)}{N}L\delta.
\label{eq:linearized_vector_form}
\end{align}
The Laplacian of the complete graph has eigenvalues
\[
0=\lambda_1<\lambda_2=\cdots=\lambda_N=N,
\]
where the zero eigenvalue corresponds to the neutral direction
$\delta\propto \mathbf{1}$ associated with uniform phase shifts.
{Accordingly, local stability is understood on the quotient by this rotational symmetry, or equivalently on the subspace orthogonal to $\mathbf{1}$.}

All transverse modes therefore decay exponentially. More precisely, if
$\eta_m(t)$ denotes the amplitude of an eigenmode with $\lambda_m>0$, then
\begin{align}
\eta_m(t)
=
\eta_m(0)\exp\!\left(
-\frac{K\,S(\theta^\ast)}{N}\lambda_m t
\right).
\label{eq:eigenmode_decay_identical}
\end{align}
For the complete graph, this yields the decay rate
\[
\frac{K\,S(\theta^\ast)}{N}\lambda_2
=
K\,S(\theta^\ast).
\]

Thus, {after removing the neutral rotation mode associated with uniform phase shifts,} the synchronized manifold is locally exponentially {stable} whenever $S(\theta^\ast)>0$, which holds for the smooth gate with finite $k$. 
The receiver gate therefore leaves the local stability mechanism unchanged in structure, but rescales the effective coupling strength by the factor $S(\theta^\ast)$. In particular, synchrony is approached more slowly when the common phase lies near the dead zone, where $S(\theta^\ast)$ is small.

This provides a simple interpretation of the role of receiver gating in the
identical-frequency setting: the dead zone does not create a new local
instability, but it can substantially delay convergence by weakening the
restoring force near synchronization.

\subsection{Numerical study of convergence and dead-zone effects}
The analytical results above show that the identical-oscillator receiver gated model preserves the classical tendency toward synchronization while introducing a phase dependent reduction in the effective coupling strength. 
We now briefly show this effect numerically. 

Figure~\ref{fig:2}A-B show how the transient approach to synchrony depends on the dead zone width. 
For small or moderate dead zones, trajectories still converge rapidly toward the synchronized state, consistent with the local exponential stability analysis. 

\begin{figure}[ht]
    \centering
    \includegraphics[width=\linewidth]{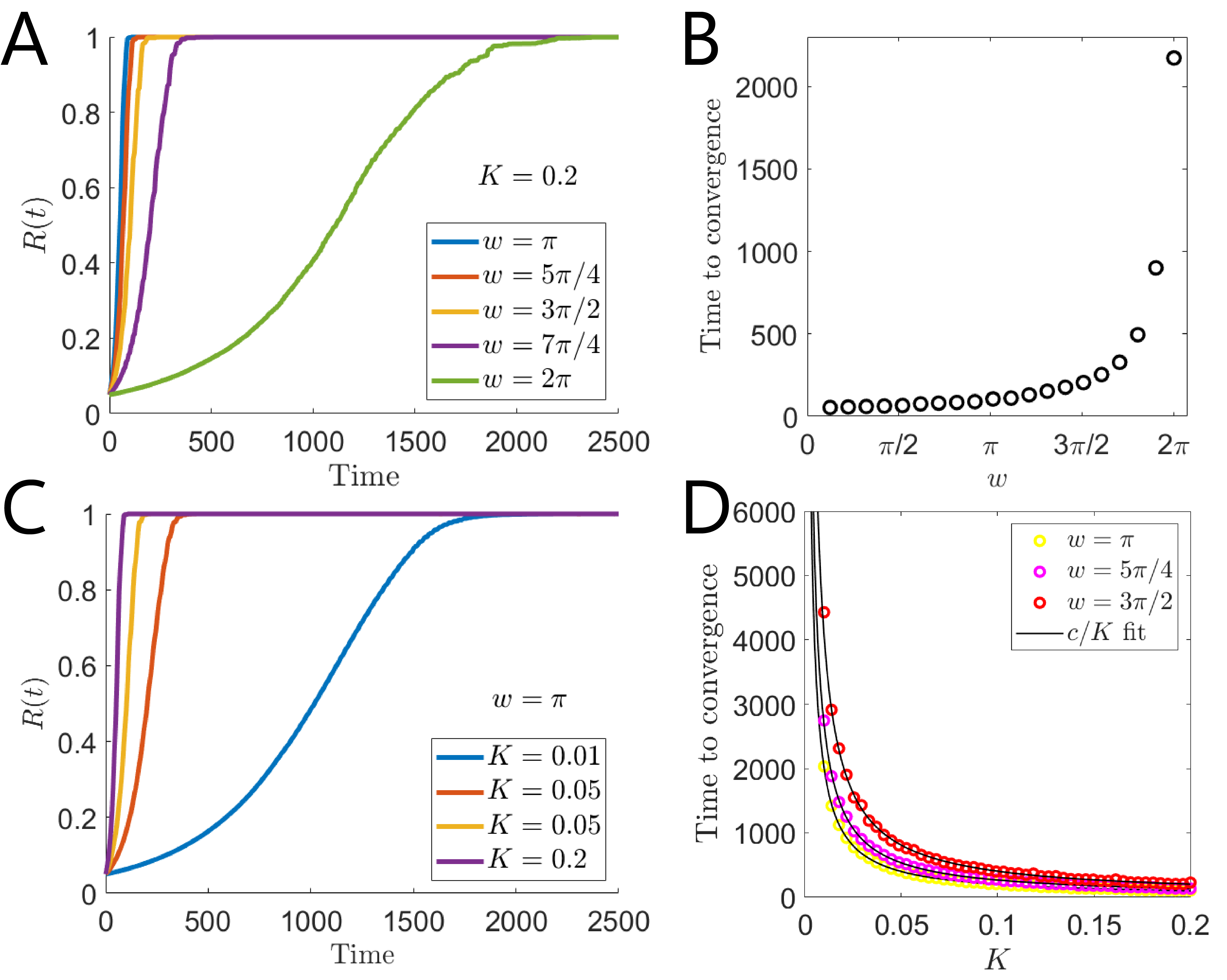}
    \caption{\textbf{Time to convergence depends nonlinearly on dead zone width.} (A) $R(t)$ over time for $w=\pi,\;5\pi/4,...,\;2\pi$. Other parameters are fixed at $K=0.2$, $N=20$, $k=10$. (B) Numerically computed time to convergence, defined as the first time point at which $|1-R(t)|<10^{-4}$, dependent on dead zone width $w$. $K=0.02,\;N=20,\;k=10$. (C) $R(t)$ over time for various $K$. Other parameters are fixed at $w=\pi$, $N=20$, $k=10$. (D) Time to convergence for varying $K$. Analytic $c/K$ curves were fit to data for $w=\pi,\;5\pi/4,\;3\pi/2$.}
    \label{fig:2}
\end{figure}

As the dead zone width increases, however, the convergence becomes noticeably slower. 
This delay is reflected both in the persistence of phase differences and in the slower growth of the order parameter toward $R=1$. 

This behavior can be understood directly from the linearization. 
Near a synchronized phase $\theta^\ast$ the characteristic decay rate is proportional to $KS(\theta^\ast)$. 
{Accordingly, the synchronization time is expected to scale heuristically like
\begin{align*}
    T_{\mathrm{sync}} \sim \frac{1}{K\,S(\theta^\ast)}.
\end{align*}}
{Thus, when the synchronized} phase lies closer to the dead zone, the gate value is smaller, the effective attraction toward synchrony is weakened, {and the transient time to synchronization increases. This heuristic is consistent with the numerical trend in Figure~\ref{fig:2}A-B, where broader dead zones produce markedly slower convergence. } 
In this way, increasing the width of the dead zone does not prevent synchronization in the identical oscillator case, but it can substantially prolong the transient time required to reach it. {Additionally, in numerical simulations, the time to convergence is proportional to $1/K$ (see \Cref{fig:2}C-D), corroborating our analytic prediction of $T_{\mathrm{sync}}$. As the dead zone width $w$ increases the time-to-convergence curve shifts upward.}


These simulations therefore complement our theoretical analysis in which the Lyapunov structure explains why synchrony remains the dominant stable outcome, while the linearization clarifies why convergence slows as dead zone effects become stronger. 

\section{Heterogeneous oscillators: phase locking and phase separation}
\label{sec:heterogeneous_analysis}
We now turn to the case of heterogeneous natural frequencies, where the oscillators need not evolve toward complete synchrony. 
In this setting, the relevant notion of collective coordination is not exact phase equality, but rather phase locking: a regime in which all oscillators rotate with a common frequency while maintaining nontrivial phase offsets. 
For the receiver gated model this heterogeneous regime is especially interesting because the phase-dependent receptivity modifies the usual balance laws of the classical Kuramoto system. 
As a result, both the locked frequency and the locked phase profile depend on the collective state of the network, and persistent phase separation may occur even when frequency locking is achieved. 

In the heterogeneous setting, phase-locked solutions rotate with a common collective frequency, so it is convenient to express the dead zone relative to a common reference phase rather than a fixed phase origin. 
Let \(\psi(t)\) denote an auxiliary reference phase defining the co-rotating frame. It is chosen so that, in a phase-locked state, the relative phases \(\theta_i(t)-\psi(t)\) remain constant in time. 
Thus \(\psi(t)\) serves as a rotating phase origin attached to the coherent collective motion, rather than necessarily being identified with the mean phase or the phase of a particular oscillator.

We therefore write the heterogeneous receiver-gated Kuramoto model as 
\begin{align}\label{eq:hetero_receiver_model_full}
\dot{\theta}_i = \omega_i+\frac{K}{N}S\bigl(\theta_i - \psi (t)\bigr)\sum_{j=1}^N \sin(\theta_j-\theta_i), \qquad i=1,\dots,N,
\end{align}
where $\omega_1, \dots, \omega_N \in \mathbb R$ are not assumed equal.
{In this formula, the dead zone is fixed in the co-rotating frame determined by $\psi(t)$ so the gate depends on the oscillator's phase relative to the common reference rather than on an absolute laboratory-frame phase. }

\subsection{Phase-locked solutions and fixed-point equations}
For heterogeneous oscillators, complete synchrony is generally no longer expected. 
Instead one seeks solutions in which all oscillators rotate with a common frequency while maintaining fixed phase offsets {relative to the reference phase.
{To describe such states, let $\Omega \in \mathbb R$ denote the common locked frequency and}
write
\begin{align}\label{eq:locked_reference_phase}
    \psi(t) = \Omega t + \psi_0.
\end{align}
{We then} introduce co-rotating coordinates
\begin{align}\label{eq:relative_phase_variables}
\phi_i(t):=\theta_i(t)-\psi(t).
\end{align}
Since $\theta_j - \theta_i = \phi_j - \phi_i$, system~\eqref{eq:hetero_receiver_model_full} becomes
\begin{align}\label{eq:hetero_relative_model}
    \dot \phi_i = \omega_i - \Omega + \frac{K}{N} S(\phi_i) \sum_{j=1}^N \sin (\phi_j - \phi_i), \qquad i = 1, \dots, N. 
\end{align}
A phase-locked state is therefore an equilibrium of the relative system~\eqref{eq:hetero_relative_model}. 
That is, we say the system is phase locked if there exist $\Omega \in \mathbb R$ and a profile $\vartheta = (\vartheta_1, \dots, \vartheta_N)\in \mathbb T^N$ such that
\begin{align}\label{eq:locked_relative_equilibrium}
    \phi_i(t) \equiv \vartheta_i, \qquad i=1, \dots, N.
\end{align}
Equivalently, in the original variables, 
\begin{align}\label{eq:locked_ansatz}
\theta_i(t)=\Omega t+ \psi_0 +  \vartheta_i,
\qquad i=1,\dots,N.
\end{align}
}
Substituting the equilibrium condition \(\phi_i(t)\equiv\vartheta_i\) into \eqref{eq:hetero_relative_model} yields
\begin{align}\label{eq:locked_equations}
\Omega = \omega_i+\frac{K}{N}S(\vartheta_i)\sum_{j=1}^N \sin(\vartheta_j-\vartheta_i), \qquad i=1,\dots,N.
\end{align}
Equivalently, 
\begin{align}\label{eq:locked_balance_equations}
\omega_i-\Omega = -\frac{K}{N}S(\vartheta_i)\sum_{j=1}^N \sin(\vartheta_j-\vartheta_i), \qquad i=1,\dots,N.
\end{align}
This is the finite-dimensional nonlinear system that determines phase locked solutions of the receiver gated model {in the co-rotating frame}. 

Equation~\eqref{eq:locked_balance_equations} has the same general interpretation as in the classical Kuramoto model: the mismatch between the intrinsic frequency $\omega_i$ and the common locked frequency $\Omega$ must be balanced by the net coupling received by oscillator $i$. 
The crucial difference is that the coupling received by each oscillator is weighted by the factor $S(\vartheta_i)$, which depends on the {oscillator's phase relative to the rotating dead-zone reference.} 
Thus the sensitivity of each oscillator to the population enters directly into the locking condition. 

This state dependence has an important consequence. In the classical Kuramoto model, the balance equations depend only on phase differences. 
In the receiver gated model, however, the individual phase locations {relative to the dead zone} also matter because the gate is evaluated at each $\vartheta_i$. 
Consequently, two locked states with the same pairwise phase differences but different placement relative to the dead zone need not have the same effective coupling structure. 

\subsection{Determination of the locked frequency}
In the classical all-to-all Kuramoto model, the locked frequency of any phase-locked state is simply the average natural frequency 
\begin{align*}
    \Omega=\frac{1}{N}\sum_{i=1}^N \omega_i,
\end{align*}
because the sine coupling cancels exactly when summed over all oscillators. 
For the receiver gated model, this cancellation no longer holds in general.
{Indeed, receiver gating breaks the usual pairwise cancellation because the interaction weight depends on the receiving oscillator, not only on the phase difference.}

Summing \eqref{eq:locked_equations} over $i$ gives
\begin{align}\label{eq:omega_sum_identity}
N\Omega = \sum_{i=1}^N \omega_i + \frac{K}{N} \sum_{i=1}^N S(\vartheta_i)\sum_{j=1}^N \sin(\vartheta_j-\vartheta_i).
\end{align}
The last term does not vanish automatically, {since the receiver-dependent factor $S(\vartheta_i)$ breaks the classical pairwise anti-symmetry.} 

Using $\sin(\vartheta_j-\vartheta_i)=-\sin(\vartheta_i-\vartheta_j),$ the double sum can be rewritten in pairwise form as 
\begin{align}\label{eq:pairwise_locked_sum}
\sum_{i=1}^N S(\vartheta_i)\sum_{j=1}^N \sin(\vartheta_j-\vartheta_i) = \sum_{1\le i<j\le N} \bigl(S(\vartheta_i)-S(\vartheta_j)\bigr) \sin(\vartheta_j-\vartheta_i).
\end{align}
Therefore,
\begin{align}\label{eq:locked_frequency_formula}
\Omega = \frac{1}{N}\sum_{i=1}^N \omega_i + \frac{K}{N^2} \sum_{1\le i<j\le N} \bigl(S(\vartheta_i)-S(\vartheta_j)\bigr) \sin(\vartheta_j-\vartheta_i).
\end{align}
This identity makes the role of receiver gating explicit. 
The locked frequency is no longer determined solely by the mean intrinsic frequency; it also depends on how the locked phase profile is positioned relative to the dead zone.
In particular, oscillators at phases where the gate is large contribute more strongly to the coupling balance than oscillators at phases where the gate is small. 
Thus, the locked frequency becomes a genuinely state-dependent quantity.

When $S$ is constant, the correction term in \eqref{eq:locked_frequency_formula} vanishes and one recovers the classical Kuramoto result. 
The receiver-gated model therefore differs from the classical
system not only in its transient dynamics, but already at the level of the frequency-balance law for locked states.

\subsection{Choice of reference phase and normalization}
In the present formulation, the quantities \(\vartheta_i\) are phases measured relative to the co-rotating reference phase \(\psi(t)\). 
Accordingly, the role of normalization here is not to quotient out a classical rotation symmetry of the original system, but rather to fix the phase origin of the rotating frame. A convenient choice is
\begin{align}\label{eq:grounding_condition}
\sum_{i=1}^N \vartheta_i = 0,
\end{align}
{which selects a representative of the locked profile by fixing the phase origin of the co-rotating frame.}
Equivalently, one may fix a single component, for example $\vartheta_N =0$.
Under such a normalization, the system \eqref{eq:locked_balance_equations} becomes a finite-dimensional nonlinear problem for the unknowns $(\vartheta,\Omega)$.

This {normalization} is useful both analytically and numerically. 
Analytically, it selects {a unique representative of the locked profile in the chosen co-rotating frame.}
Numerically, it prevents spurious drift along the {choice of phase origin} 
and makes the computation of locked states better posed.
{It is also consistent with the numerical visualizations, where phase offsets are plotted relative to a common reference, taken in Figure~\ref{fig:3} to be the mean phase \(\bar{\theta}\).}

\begin{figure}[H]
    \centering
    \includegraphics[width=\linewidth]{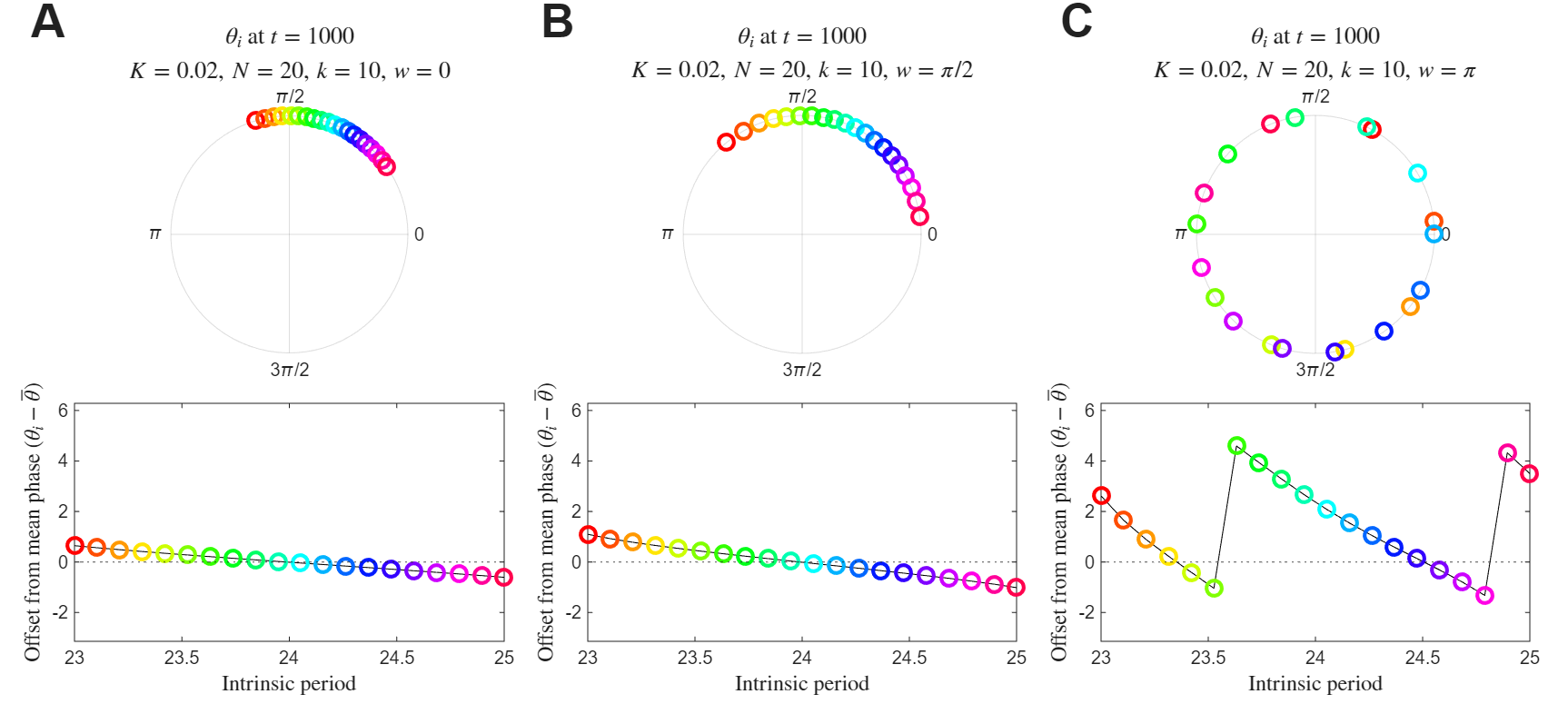}
    \caption{
\textbf{Solutions at $t=1000$.} (A) Phase-locked phases of oscillators at $t=1000$ for $K=0.02$, $N=20$, and no dead zone. Initial conditions were uniformly distributed from $0$ to $2\pi$. Below, phase offset of oscillators at $t=1000$ relative to mean phase $\overline{\theta}$ plotted against each oscillator's intrinsic period $\tau = 2\pi/\omega_i$. (B) Phase-locked phases of oscillators at $t=1000$ for $K=0.02$,\; $N=20$,\; $w=\pi/2$. (C) Non-phase-locked solutions at $t=1000$ for $K=0.02$,\; $N=20$,\; $w=\pi$.}
    \label{fig:3}
\end{figure}


{A key difference from the identical-frequency case is that, near a nontrivial locked state, the coupling sums need not vanish. As a result, the derivative of the receiver gate generally contributes at leading order in the heterogeneous linearization.}

\subsection{Linearization about a phase-locked profile}
We now examine the local dynamics near a phase-locked state of the relative system
\begin{align}\label{eq:hetero_relative_model_linearization}
    \dot \phi_i = \omega_i - \Omega + \frac{K}{N} S(\phi_i) \sum_{j=1}^N \sin(\phi_j - \phi_i), \qquad i=1, \dots, N. 
\end{align}
Let $\vartheta = (\vartheta_1, \dots, \vartheta_N) \in \mathbb T^N$ be a phase-locked profile satisfying
\begin{align}\label{eq:locked_balance_equations_linearization}
    \omega_i - \Omega = -\frac{K}{N} S(\vartheta_i) \sum_{j=1}^N
 \sin(\vartheta_j - \vartheta_i), \qquad i=1,\dots, N. 
 \end{align}
To study perturbations about this coherent state, we write
\begin{align}\label{eq:hetero_perturbation_ansatz}
    \phi_i(t) = \vartheta_i + \eta_i(t), \qquad i=1,\dots, N,
\end{align}
where $|\eta_i(t)| \ll 1$. Substituting \eqref{eq:hetero_perturbation_ansatz} into \eqref{eq:hetero_relative_model_linearization}, expanding to first order, and using \eqref{eq:locked_balance_equations_linearization}, we obtain
\begin{align}\label{eq:hetero_linearized_component_form}
\dot{\eta}_i &= \frac{K}{N} \Bigg[S'(\vartheta_i)\eta_i
\sum_{j=1}^N \sin(\vartheta_j-\vartheta_i) + S(\vartheta_i) \sum_{j\ne i}\cos(\vartheta_j-\vartheta_i)\,(\eta_j-\eta_i) \Bigg] + \mathcal O(|\eta|^2),
\end{align}
where $|\eta|^2$ denotes terms that are at least quadratic in the perturbation. 

Equation~\eqref{eq:hetero_linearized_component_form} makes the structure of the heterogeneous linearization transparent. 
The first term arises from the variation of the receiver gate itself and is proportional to $S'(\vartheta_i)$. 
The second term is the usual coupling contribution obtained by linearizing the sine interaction, but it is weighted by the local gate value $S(\vartheta_i)$. 
Thus, both the value and the derivative of the receiver gate enter the local dynamics near a non-trivial locked profile. 

Equivalently, the linearized system may be written in the matrix form as
\begin{align}\label{eq:hetero_linearized_matrix_form}
    \dot \eta = J(\vartheta) \eta + \mathcal O(|\eta|^2)
\end{align}
where the Jacobian $J(\vartheta) = \left[J_{i\ell}(\vartheta)\right]_{i,\ell=1}^N$ is given by
\begin{align}\label{eq:hetero_jacobian_offdiag}
    J_{i\ell}(\vartheta) = \frac{K}{N} S(\vartheta_i) \cos(\vartheta_\ell - \vartheta_i), \qquad \ell \neq i,
\end{align}
\begin{align}\label{eq:hetero_jacobian_diag}
    J_{ii}(\vartheta) = \frac{K}{N} \left(S'(\vartheta_i) \sum_{j=1}^N \sin(\vartheta_j - \vartheta_i) - S(\vartheta_i) \sum_{j \neq i} \cos(\vartheta_j - \vartheta_i)\right). 
\end{align}
Using the locked state balance relation~\eqref{eq:locked_balance_equations_linearization}, the diagonal entry may also be written as
\begin{align}\label{eq:hetero_jacobian_diag_alt}
    J_{ii}(\vartheta) = \frac{S'(\vartheta_i)}{S(\vartheta_i)} (\Omega - \omega_i) - \frac{K}{N} S(\vartheta_i) \sum_{j \neq i} \cos(\vartheta_j - \vartheta_i), 
\end{align}
which makes the role of heterogeneity especially clear. 
Indeed, the term involving $S'(\vartheta_i)$ is weighted by the mismatch between the locked frequency and the intrinsic frequency of oscillator $i$. 
Hence, whenever the oscillators have different natural frequencies and the locked profile is non-trivial, the derivative of the receiver gate contributes directly to the local linear dynamics. 


{This is an important difference from the identical-frequency synchronization ana-lysis. In the identical case, linearization about complete synchrony eliminates the derivative of the gate at first order because the coupling sum vanishes at the synchronized state. Here, by contrast, the sums \(\sum_{j=1}^N \sin(\vartheta_j-\vartheta_i)\) need not vanish, so \(S'(\vartheta_i)\) enters directly into the Jacobian. Thus, the local structure of a heterogeneous locked state depends not only on the phase offsets themselves, but also on how sensitively each oscillator responds near its location relative to the dead zone.}

The matrix $J(\vartheta)$ provides the natural local stability test for a phase locked profile. 
If all eigenvalues of $J(\vartheta)$ have negative real part (after imposing the chosen normalization of the rotating frame, if applicable), then the corresponding phase locked state is locally asymptotically stable in the relative coordinates. 
In general, because the Jacobian is neither symmetric nor circulant, its spectral properties depend non-trivially on the locked phase profile, the intrinsic frequencies, and the gate shape. 
For this reason, the formulas above are most useful both as a structural analytical characterization and as the basis for numerical stability calculations of coherent states in the heterogeneous receiver gated model.

\allblack

\subsection{Phase separation in heterogeneous oscillators}
A central feature of the heterogeneous receiver-gated model is that frequency locking does not imply complete synchrony. Even when all oscillators rotate with a common long-time frequency, the corresponding locked profile may remain distributed across the circle. We refer to this persistent spread of the locked profile as \emph{phase separation}. More precisely, phase separation occurs when the system approaches a phase-locked state
\begin{align*}
    \phi_i(t)\equiv \vartheta_i, \qquad i=1,\dots,N,
\end{align*}
with nontrivial relative phase offsets, so that \(\vartheta_i-\vartheta_j\not\equiv 0 \;(\mathrm{mod}\;2\pi)\) for at least some pairs \((i,j)\). In such a state, the oscillators share a common locked frequency \(\Omega\), but they do not collapse to a single phase.

To quantify the amount of phase separation in a locked profile \(\vartheta=(\vartheta_1,\dots,\vartheta_N)\), we define
\begin{align*}
    \Delta_{\max}(\vartheta):=\max_{1\le i,j\le N}\operatorname{dist}_{\mathbb T}(\vartheta_i,\vartheta_j),
\end{align*}
where \(\operatorname{dist}_{\mathbb T}(a,b)=|d(a,b)|\) denotes the circular distance on \(\mathbb T\), with \(d\) as in \eqref{eq:wrapped_distance}. Thus, \(\Delta_{\max}=0\) corresponds to complete synchrony, while larger values of \(\Delta_{\max}\) indicate broader phase-separated locked profiles.

Phase separation arises naturally in the receiver-gated model because the locking balance depends not only on pairwise phase differences, but also on the local gate values \(S(\vartheta_i)\). Oscillators with different intrinsic frequencies require different coupling corrections in order to rotate at a common locked frequency, and these corrections are modulated by phase-dependent receptivity. As a result, the coherent state reflects a compromise among intrinsic frequency mismatch, relative phase offsets, and receiver-side phase sensitivity.

Figure~\ref{fig:3} illustrates this behavior for a representative heterogeneous simulation. Panel (A) shows that the oscillator phases at late time remain distributed around the circle rather than collapsing to a single point, consistent with a phase-locked state satisfying \(\Delta_{\max}(\vartheta)>0\). In the lower panel, the corresponding phase offsets relative to the mean phase \(\bar\theta\) vary systematically with intrinsic period, indicating that the phase-separated locked profile is structured by frequency heterogeneity.

Thus, coherent collective behavior in the receiver-gated model need not take the form of near synchrony. Instead, the model naturally supports frequency-locked states with persistent and structured phase separation. 
\allblack

\subsection{A local coherent branch for small heterogeneity}
We now show that phase locked states persist under sufficiently small frequency heterogeneity. 
This provides a local analytical justification for the coherent profiles observed numerically and clarifies how phase separation emerges from the synchronized state. 

Let 
\begin{align}\label{eq:small_heterogeneity_frequencies}
    \omega_i = \bar \omega + \varepsilon \nu_i, \qquad i=1, \dots, N,
\end{align}
where $\bar \omega \in \mathbb R$ is the mean intrinsic frequency, $\nu_1, \dots, \nu_N \in \mathbb R$ satisfy
\begin{align}
    \sum_{i=1}^N \nu_i = 0,
\end{align}
and $\varepsilon$ is a small parameter measuring the strength of heterogeneity.
Fix a reference phase $\theta^\ast \in \mathbb T$ with $S(\theta^\ast)>0$. When $\varepsilon=0$, the identical frequency system admits the synchronized locked profile 
\begin{align*}
    \vartheta_i = \theta^\ast, \qquad i=1, \dots, N, 
\end{align*}
with common locked frequency $\Omega = \bar \omega$. We seek a nearby branch of phase locked solutions of the form 
\begin{align}\label{eq:rho_decomposition}
    \vartheta_i = \theta^\ast + \rho_i, \qquad i=1, \dots, N, 
\end{align}
where the perturbation $\rho = (\rho_i, \dots, \rho_N)$ is assumed small and satisfies the normalization
\begin{align*}
    \sum_{i=1}^N \rho_i =0.
\end{align*}
Thus, $\rho_i$ measures the deviation of oscillator $i$ from the synchronized reference phase $\theta^\ast$. Substituting \eqref{eq:small_heterogeneity_frequencies} and \eqref{eq:rho_decomposition} into the locked state equations
\begin{align*}
    \Omega = \omega_i + \frac{K}{N} S(\vartheta_i) \sum_{j=1}^N \sin (\vartheta_j - \vartheta_i), \qquad i=1, \dots, N, 
\end{align*}
gives the nonlinear system
\begin{align}\label{eq:small_heterogeneity_balance}
    0 = \bar \omega + \varepsilon \nu_i -\Omega + \frac{K}{N} S(\theta^\ast + \rho_i) \sum_{j=1}^N \sin (\rho_j - \rho_i), \qquad i=1, \dots, N. 
\end{align}

\begin{prop}
    Fix $\theta^\ast \in \mathbb T$ such that $S(\theta^\ast)>0$, and let $\omega_i = \bar \omega + \varepsilon \nu_i$ with $\sum_{i=1}^N \nu_i =0$. Then there exists $\varepsilon_0>0$ such that for all $|\varepsilon|< \varepsilon_0$, the receiver gated system admits a unique smooth branch of phase locked solutions $(\rho(\varepsilon), \Omega(\varepsilon))$ satisfying $\rho(0) =0, \Omega(0) = \bar \omega$, and the normalization $\sum_i \rho_i(\varepsilon)=0$.
    Equivalently, there exists a unique smooth locked profile 
    \begin{align*}
        \vartheta_i (\varepsilon) = \theta^\ast + \rho_i(\varepsilon), \qquad i=1, \dots, N, 
    \end{align*}
    near the synchronized state $\vartheta_i = \theta^\ast$. \\
    Moreover, this branch satisfies the first order expansion
    \begin{align}
        \rho_i(\varepsilon) &= \frac{\varepsilon}{K S(\theta^\ast)} \nu_i + \mathcal O(\varepsilon^2), \qquad i=1, \dots, N, \label{eq:rho_first_order_expansion}\\[0.5em]
        \Omega(\varepsilon) &= \bar \omega + \mathcal O(\varepsilon^2). \label{eq:Omega_first_order_expansion}
    \end{align}
    Hence, 
    \begin{align}\label{eq:phase_difference_first_order}
        \vartheta_i(\varepsilon) - \vartheta_j(\varepsilon) = \frac{\varepsilon}{K S(\theta^\ast)} (\nu_j - \nu_i) + \mathcal O(\varepsilon^2), \qquad i, j = 1, \dots, N. 
    \end{align}
\end{prop}

\begin{proof}
    Define 
    \begin{align*}
        \mathcal G_i (\rho, \Omega, \varepsilon) := \bar \omega + \varepsilon \nu_i -\Omega + \frac{K}{N} S(\theta^\ast +\rho_i) \sum_{j=1}^N \sin (\rho_j - \rho_i), \qquad i=1, \dots, N,
     \end{align*}
     and impose the normalization through the additional equation 
     \begin{align*}
         \mathcal G_{N+1}(\rho, \Omega, \varepsilon) := \sum_{i=1}^N \rho_i 
     \end{align*}
Thus we seek zeros of $\mathcal G = (\mathcal G_1, \dots, \mathcal G_N, \mathcal G_{N+1}) \in \mathbb R^{N+1}$. At $(\rho, \Omega, \varepsilon) = (0, \bar \omega, 0)$, one has $\mathcal G(0, \bar \omega, 0)=0$, since the sine sum vanished when all $\rho_i = 0$. We compute the derivative of $\mathcal G$ with respect to $(\rho, \Omega)$ at this point. 
For a variation $(x, \mu)\in \mathbb R^N \times \mathbb R$, 
\begin{align}
    &D_{(\rho, \Omega)} \mathcal G_i(0, \bar \omega, 0)[x, \mu] = -\mu + \frac{K}{N} S(\theta^\ast) \sum_{j=1}^N (x_j - x_i), \qquad i=1, \dots, N, \label{eq:DG_first_N} \\[0.5em]
    &D_{(\rho, \Omega)} \mathcal G_{N+1}(0, \bar \omega, 0)[x, \mu] = \sum_{i=1}^N x_i. \label{eq:DG_constraint}
\end{align}
Suppose $D_{(\rho, \Omega)} \mathcal G(0, \bar \omega, 0)[x, \mu] =0$. Summing \eqref{eq:DG_first_N} over $i$ gives $-N_\mu =0$, so $\mu=0$. Then \eqref{eq:DG_first_N} implies
\begin{align*}
    \sum_{j=1}^N x_j - Nx_i =0, \qquad i=1, \dots, N,
\end{align*}
So all $x_i$ are equal. 
Using \eqref{eq:DG_constraint}, we obtain $\sum_i x_i=0$ for all $i$. 
Therefore, the derivative $D_{(\rho, \Omega)} \mathcal G (0, \bar \omega, 0)$ is invertible. 

By implicit function theorem~\cite{krantz2002implicit}, there exist $\varepsilon_0>0$ and a unique smooth branch $(\rho(\varepsilon), \Omega(\varepsilon))$ for $|\varepsilon|< \varepsilon_0$ such that 
\begin{align*}
    \mathcal G \left(\rho(\varepsilon), \Omega(\varepsilon), \varepsilon \right) =0, \quad \rho(0)=0, \quad \Omega(0)= \bar \omega.
\end{align*}
To obtain the leading order behavior, write 
\begin{align*}
    \rho_i (\varepsilon) = \varepsilon a_i + \mathcal O(\varepsilon^2), \qquad \Omega(\varepsilon) = \bar \omega + \varepsilon \Omega_1 + \mathcal O(\varepsilon^2). 
\end{align*}
where $\Omega_1$ denotes the first-order correction to the locked frequency.
Expanding \eqref{eq:small_heterogeneity_balance} to first order gives
\begin{align}\label{eq:first_order_balance}
0 = \nu_i - \Omega_1 + \frac{K}{N}S(\theta^\ast)\sum_{j=1}^N (a_j-a_i),
\qquad i=1,\dots,N.
\end{align}
Summing over $i$ and using $\sum_i \nu_i =0$ gives $\Omega_1=0$. Hence
\begin{align}
    a_i = \frac{\nu_i}{K S(\theta^\ast)}, \qquad i=1, \dots, N, 
\end{align}
which yields \eqref{eq:rho_first_order_expansion} and \eqref{eq:Omega_first_order_expansion}. 
Formula \eqref{eq:phase_difference_first_order} follows immediately.
\end{proof}
This result gives a precise local description of phase separation near synchrony. 
For sufficiently small heterogeneity, a unique coherent branch persists, and the phase offsets are ordered at leading order by intrinsic frequency deviation $\nu_i$. 
In particular, the amount of phase separation scales like $1/S(\theta^\ast)$: if the synchronized reference phase lies in a weakly receptive part of the cycle, then the same level of frequency heterogeneity produces a broader locked profile.
Thus, even close to synchrony, the receiver gate directly controls how intrinsic heterogeneity is converted into persistent phase dispersion.

\subsection{Influence of dead zones for heterogeneous oscillators}
To further examine the role of receiver gating in the heterogeneous regime, we numerically compare the collective dynamics for various dead zone widths while keeping the remaining parameters fixed. By introducing a dead zone of width $w=\pi/2$ (see \Cref{fig:3}B) we find that phase-locked solutions are more widely distributed relative to the mean phase. 
Figure~\ref{fig:4} shows representative simulations over time for $N=20$ heterogeneous oscillators with coupling strength $K=0.02$ and gate sharpness $k=10$ as before. 
The top row shows the Kuramoto order parameter $R(t)$ while the bottom row shows the phase differences $\theta_i - \bar \theta$ relative to the mean phase.

\begin{figure}[H]
    \centering
    \includegraphics[width=1\linewidth]{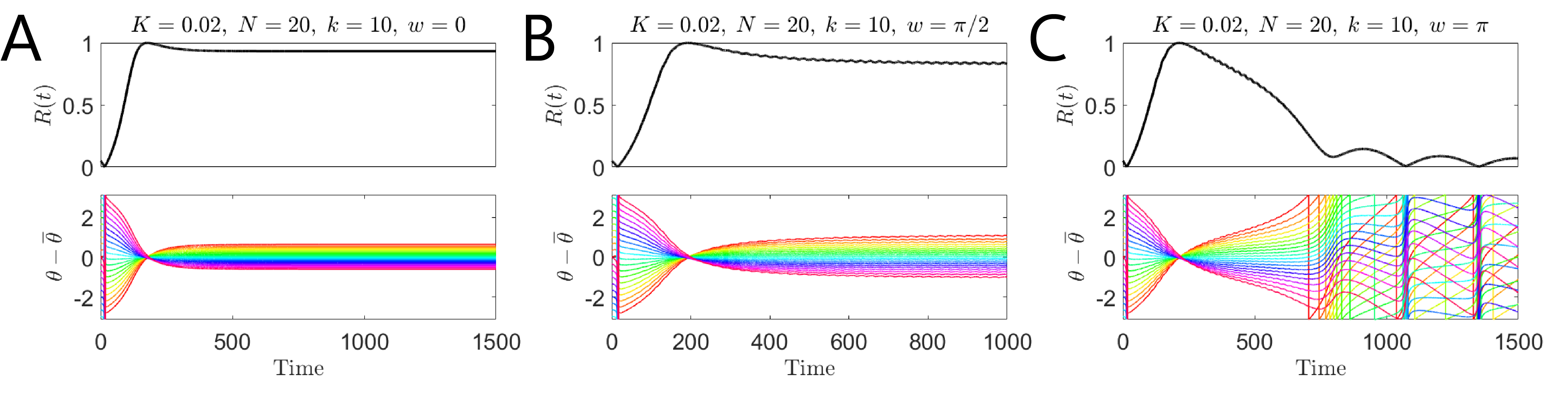}
    \caption{\textbf{Phase distribution depends on dead zone width.} $K=0.02,\; N=20,\; k=10\; \omega_i=2\pi/\left(23+\frac{2(i-1)}{N-1}\right)$. Initial conditions were uniformly distributed from $0$ to $2\pi$. (A) Numerical solutions with no dead zone. (B) Dead zone width $w=\pi/2$. (C) $w=\pi$.}
    \label{fig:4}
\end{figure}

Without a dead zone, shown in Figure~\ref{fig:4}A, the system is identical to the classical heterogeneous Kuramoto model. 
The order parameter rapidly increases toward a value close to one, indicating a highly coherent state, and the phase differences settle into a narrow bounded profile. 
Although the oscillators do not collapse to a single phase, they approach a {numerically observed} phase-locked configuration with only modest phase separation. 

When the dead zone width is increased to $w=\pi/2$ as shown in Figure~\ref{fig:4}B, the system still converges to a coherent state but the long time behavior changes in an important way, with the phase-locked profile more broadly distributed. 
Thus, frequency coordination is preserved but the gate weakens the effective coupling enough to allow larger persistent phase offsets in the locked state. 

A more dramatic change occurs for the larger dead zone $w=\pi$ shown in Figure~\ref{fig:4}C. 
In this case, the order parameter transiently rises but then declines as the oscillators disperse and drift in phase.
The corresponding phase trajectories no longer settle into a tightly organized locked profile. 
Instead, the population exhibits a much broader and more irregular phase distribution (also see \Cref{fig:3}C), indicating that receiver gating can prevent coherent locking when the dead zone is large enough. Overall, increasing the dead zone width does not simply slow convergence as in the identical frequency case. For heterogeneous oscillators, dead zones can alter the long time coherent state itself.

Furthermore, we sampled the dynamics on the Poincar\'{e} section $\bar \theta = \pi \;(\mathrm{mod}\;2\pi)$, where $\bar \theta = \arg\!\left(\frac{1}{N}\sum_{j=1}^N e^{i\theta_j}\right)$ denotes the mean phase of the population; see \Cref{fig:5}. This choice provides one snapshot per collective rotation and allows trajectories with different long-time behaviors to be compared on a common phase slice. 
We plot the relative phase $\theta_1-\bar{\theta}$ rather than the absolute phase $\theta_1$, since phase locking is a property of phase differences and not of the overall rotational position of the population. 
In the phase-locked state, the relative phases become constant, so the Poincar\'{e} section collapses to a single point (or a small numerically broadened cluster) after a transient period. 
By contrast, if the relative phases continue to evolve, then successive intersections produce a cloud or band of points, indicating phase drifting.

\begin{figure}[H]
    \centering
    \includegraphics[width=\linewidth]{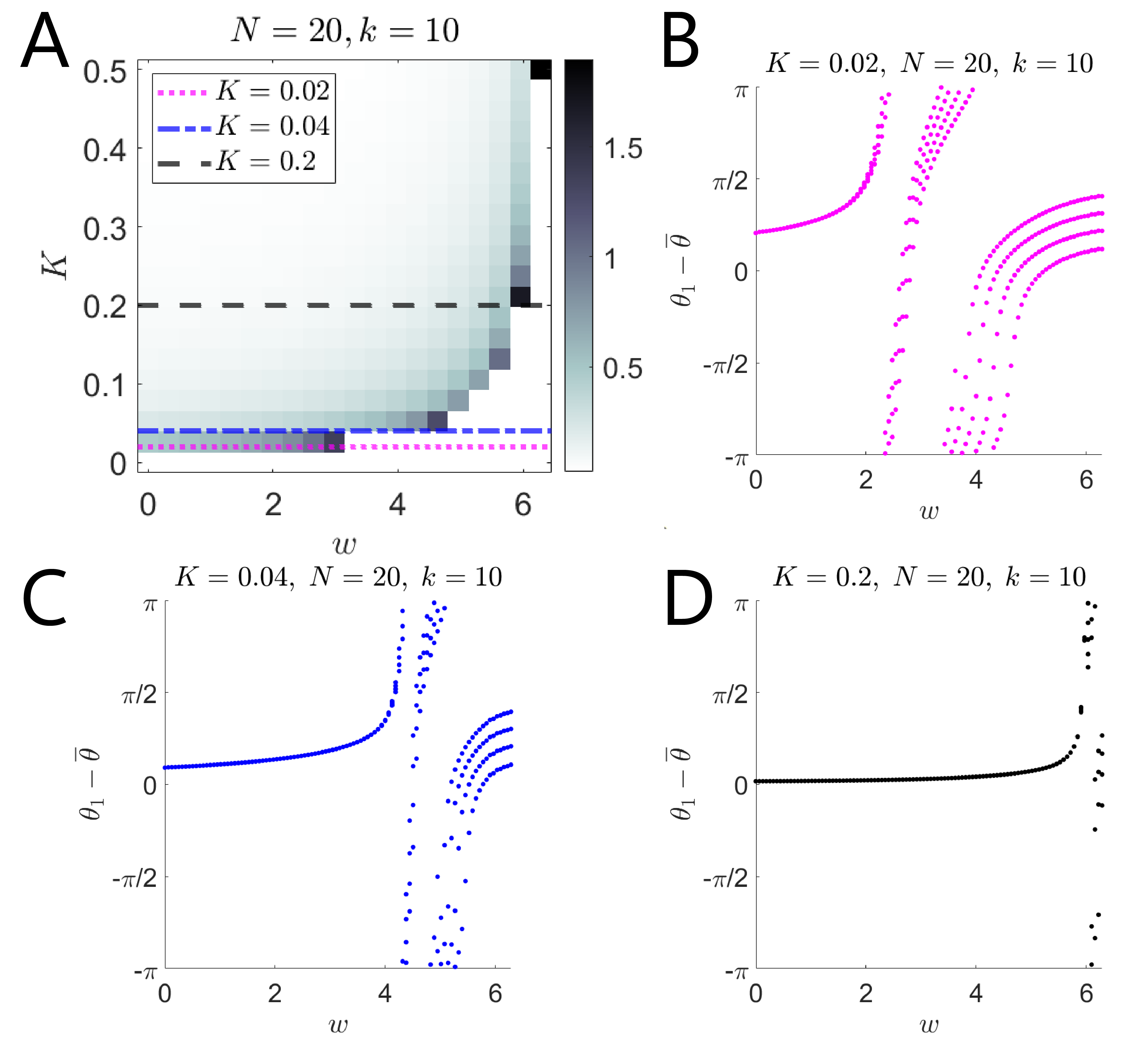}
    \caption{\textbf{Numerical Poincar\'{e} sections at mean phase sampled at $\overline{\theta} = \pi \;(\text{mod}\; 2\pi)$.} (A) Steady phase gaps $\theta_1-\overline{\theta}$ for varying coupling strength $K$ and dead zone width $w$. Numerical solutions were computed to $t=1000$ with initial conditions distributed uniformly from 0 to $2\pi$. Poincar\'{e} sections of the phase gap $\theta_1-\overline{\theta}$ were taken at $\overline{\theta} = \pi$ after $t=900$ to account for transient behavior.
    {Isolated, smooth parts of the curves in panels B-D correspond to phase-locked behavior, whereas extended point clouds indicate phase drifting.}
    The lower right  region corresponds to non-phase-locked behavior. (B) Poincar\'{e} sections for $K=0.02$. (C) Poincar\'{e} sections for $K=0.04$. (D) Poincar\`{e} sections for $K=0.2$.} 
    \label{fig:5}
\end{figure}

As discussed earlier, the dead zone effectively increases the critical coupling strength $K$ required to achieve long-term phase-locked solutions. This dependence is nonlinear, as shown in \Cref{fig:5}. 
For fixed $K=0.02$, the width of the phase distribution increases with $w$ until the system
{undergoes a numerically observed transition near }
$w=1.7$,
{beyond which}
the steady phase gaps
{no longer remain localized}
and solutions are no longer phase-locked; see \Cref{fig:5}B.

Thus, for weakly coupled oscillators, dead zones can have a substantial impact on synchronization dynamics, allowing for flexibility in the long-term phase distribution and, when large enough, preventing phase-locking entirely. 
Increasing the coupling strength $K$ shifts 
{this numerically observed transition,}
allowing 
long-term phase-locked solutions 
{to persist even for}
large $w$; see \Cref{fig:5}C-D. Interestingly, strong enough coupling eventually masks the dead zone's effects.

Earlier, the fixed-point equations and locked-frequency identity showed that the heterogeneous receiver gated model depends sensitively on the phase profile through the gate values $S(\vartheta_i)$. 
Figure~\ref{fig:4} demonstrates the dynamical consequence of this dependence; as the dead zone widens, the network can move from strong coherent locking to broader phase separation and, eventually, to incoherence.

To examine the robustness of our results, we tested the sensitivity of the dynamics to $k$ and $N$ which were fixed in the numerical simulations above. The parameter $k$ determines the steepness of the gating function at transitions into and out of the dead zone; see \Cref{fig:S1} Panel (A). We found that the effect of the dead zone on model dynamics is generally robust to variations in $k>0.5$. In \Cref{fig:S1} we show numerically that the dependence of steady phase gaps on $w$ is identical or nearly identical for a wide range of $k$. Small values of $k$ (e.g., $k<0.5$) significantly restrict inactive windows and don't create much of a dead zone as shown in \Cref{fig:S1} Panel (A). We found that the dependence of model dynamics on $w$ remains robust to increasing $k$. Thus, in the parameter space near $k=10$ used throughout the paper, $k$ is not an important determinant of long-term dynamics. Similarly, we found that the dynamics did not depend meaningfully on large enough $N$, as in the classical Kuramoto model. \Cref{fig:S1} Panel (C) compares the dynamics of $N=20$ oscillators with $N=200$, which are qualitatively alike. Over a wide range of $N$ (see \Cref{fig:S1} Panel (D)), the dependence of the model dynamics on $w$ remains consistent.

\section{Discussion}
In this paper, we introduced a receiver-gated Kuramoto model in which the effective coupling is modulated by the phase of the receiving oscillator through a smooth dead-zone gate. 
This provides {a biologically motivated extension of} 
the classical {all-to-all} Kuramoto system 
{in which interactions are shaped not only by phase differences, but also by phase-dependent receptivity.}
{A central contribution of the model is that it preserves enough analytical structure to permit rigorous study, while at the same time generating collective behaviors that differ qualitatively from those of the classical system.}

{Our analytical results in the identical-frequency regime show that receiver gating modifies the dynamics without altering the basic synchronization mechanism.}
In particular, {the model retains a Lyapunov-type structure, and because the gate remains strictly positive for finite $k$, the classical critical set of}
Kuramoto potentials is preserved. 
{Thus, the receiver gate changes the rate of descent toward equilibrium rather than the equilibrium set itself.
The local stability analysis further shows that, after removing the neutral rotational mode, the transverse decay rate near synchrony is proportional to $K S(\theta^\ast)$.
In the identical-oscillator setting, synchrony remains the distinguished stable coherent state, but convergence toward synchrony can become substantially slower when the synchronized phase lies in a weakly receptive portion of the cycle. }


The heterogeneous-frequency case reveals the deeper structural effect of receiver gating.
{At the level of the model equations, the locked-frequency relation is no longer determined solely by the average intrinsic frequency.}
Instead, {because the interaction weight depends on the receiver, the usual pairwise cancellation underlying the classical Kuramoto balance law breaks down, and the common locked frequency acquires an explicit dependence on the locked phase configuration.}
{This is an exact consequence of the receiver-gated coupling law and represents one of the main mathematical differences between the present model and the standard heterogeneous Kuramoto system.}

A second {analytical contribution} 
in the heterogeneous setting is the {local coherent branch obtained for sufficiently small frequency heterogeneity.}
Near synchrony, we showed that a unique smooth branch of phase-locked solutions persists, and that the leading-order phase offsets are ordered by the intrinsic frequency deviations. 
This result provides a local mathematical description of how phase separation emerges from the synchronized state. 
In particular, the leading-order scaling shows that the width of the locked phase profile is amplified by the factor $1/S(\theta^\ast)$. Thus, even close to synchrony, the receiver gate directly controls how intrinsic heterogeneity is converted into persistent phase dispersion.

The numerical results complement these analytical findings by revealing how the coherent states are organized away from the perturbative regime. In the identical-frequency case, the simulations confirm that increasing the dead zone width slows the transient approach to synchrony, consistent with the linearized decay rate. In the heterogeneous case, the numerics reveal richer dynamics. For small or moderate dead zones, stable phase-separated profiles are shaped jointly by intrinsic heterogeneity and the receiver gate, so that coherent collective behavior may take the form of an organized but dispersed locked state rather than a tightly clustered one.

For sufficiently large dead zones in the parameter regimes explored here, the numerical Poincar\'e sections indicated a transition from localized phase-locked behavior to phase drifting. We emphasize that this last conclusion is based on numerical evidence rather than a formal bifurcation theorem. Nevertheless, the computations strongly suggest that receiver gating can do more than slow synchronization or broaden a locked state--it can also destabilize coherent locking when phase-dependent receptivity becomes sufficiently weak over a large portion of the cycle.

Several limitations should be noted. We considered only all-to-all coupling and a smooth phenomenological gate, whereas biological oscillator populations may have more complex connectivity and phase-response structure. 
In addition, the model is not calibrated to experimentally measured PRCs or sensitivity functions so the biological interpretation remains generic. 

Biologically observed phase heterogeneity is thought to enable flexible adaptation of biological rhythms to external cues and is influenced by heterogeneity in individual cellular properties \cite{merrow2020functional}. Our study suggests that this heterogeneity can also be modulated by phase-dependent gating mechanisms in oscillator coupling, driving further heterogeneity in observed period.

Overall, the receiver-gated Kuramoto model shows that incorporating phase-dependent receptivity into oscillator coupling can preserve classical synchronization structure in homogeneous population while producing qualitatively new coherent states in heterogeneous ones. 
In particular, it provides a simple mechanistic framework in which phase locking, phase separation, and dead-zone-induced phase drifiting can all arise from the same biologically motivated modification of the coupling law. 

\appendix

\begin{appendices}
\section{Model parameters and sensitivity}

\begin{table}[H]
\centering
\caption{Parameters and values used throughout the paper.}
\begin{tabular}{lll}
\hline
Parameter  & Value(s)                                             & Description                           \\ \hline
$N$        & $20$                                                 & Number of oscillators                 \\
$K$        & $0.02,\; 0.2,\; [0,0.5]$                           & Coupling strength                     \\
$w$        & $(0,2\pi)$                                         & Dead zone width                       \\ 
$k$        & $10,\; [0,10]$                                       & Gating function steepness             \\
$\omega_i$ & $\frac{2\pi}{24},\; \frac{2\pi}{\left(23+\frac{2(i-1)}{N-1}\right)}$ & Intrinsic frequency of oscillator $i$
\end{tabular}
\label{table:S1}
\end{table}

\setcounter{figure}{0} 
\renewcommand{\thefigure}{A\arabic{figure}}

\begin{figure}[H]
    \centering
    \includegraphics[width=\linewidth]{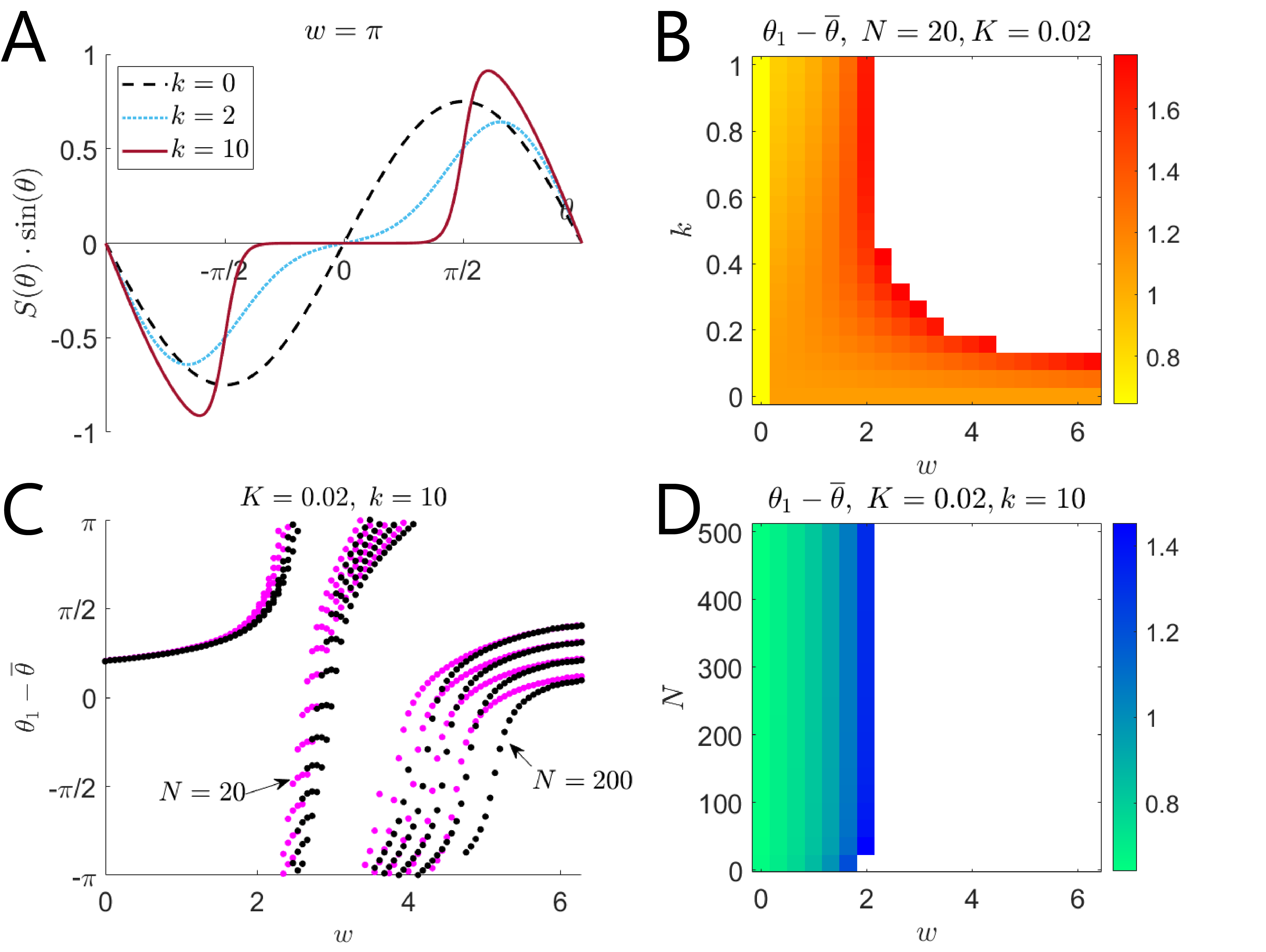}
    \caption{\textbf{Robustness of model behavior to variation in $k$ and $N$.} (A) Slice of receiver-gated coupling function for gating function steepness $k=0$, $k=2$, and $k=10$. (B) Dependence of steady phase gaps $\theta_1-\overline{\theta}$ on dead zone width $w$ is robust to changes in $k>0.5$. We do not consider gating with small $k$ as effective dead zones. (C) Numerical poincar\'{e} sections of the phase gap $\theta_1-\overline{\theta}$ sampled at $\overline{\theta}=\pi$ are qualitatively close for $N=20$ (pink) and $N=200$ (black). (D) Dependence of steady $\theta_1-\overline{\theta}$ on dead zone width $w$ is generally robust to the number of oscillators $N$.}
    \label{fig:S1}
\end{figure}
\end{appendices}

\section*{Acknowledgments}
The authors would like to thank Victoria Booth for carefully reading the manuscript and providing helpful comments. N.A. was supported by NIH/NCCIH grant R01AT013188. {R.K. was supported by NIH grant T90 DE034663. 

\section*{Author Contributions}
All authors have made substantial intellectual contributions to the study conception, 
execution, and design of the work. All authors have read and approved the final manuscript.

\section*{Access to Code}
All numerical solutions presented throughout this paper were computed using MATLAB's ode45 function with code made publicly available through our \href{https://github.com/rubyshkim/AkhavanKim_DeadZone}{Github repository} upon acceptance of the manuscript.

\bibliographystyle{ieeetr}
\bibliography{References.bib}

\end{document}